\documentclass[11pt,reqno]{amsart}
\usepackage{amsrefs}
\usepackage{amsfonts, amsmath, amssymb, amscd, amsthm, bm}
\usepackage{cases}
\usepackage{enumerate}

\usepackage{comment}

\usepackage{enumerate}
\usepackage{amsfonts, amsmath, amssymb, amscd, amsthm, bm}

\usepackage{enumerate}
\usepackage{url}
\usepackage[linktocpage=true,colorlinks,citecolor=magenta,linkcolor=blue,urlcolor=magenta]{hyperref}
\usepackage{multicol}
\usepackage{comment}

\newtheorem{thm}{Theorem}[section]

\newtheorem{lem}[thm]{Lemma}

\theoremstyle{definition}

\theoremstyle{theorem}
\newtheorem{rem}[thm]{Remark}

\theoremstyle{claim}

\numberwithin{equation}{section}
\def\F{{\mathbb M}^{N}_c}
\def\H{\mathbb{H}^{N}(-1)}
\def\S{\mathbb{S}^{N}(1)}
\def\R{\mathbb{R}^{N}}
\def\et{\tilde{e}}
\def\th{\tilde{h}}
\def\tH{\tilde{H}}
\def\tg{\tilde{g}}
\def\tbg{\tilde{\bar{g}}}
\def\tom{\tilde{\omega}}
\def\tt{\tilde{\theta}}

\def\tR{\tilde{R}}

\def\la{\lambda}

\def\pl{\Delta_p}
\def\l{\langle}
\def\r{\rangle}

\DeclareMathOperator\vol{vol}
\DeclareMathOperator\dive{div}

\DeclareMathOperator\arcsinh{arcsinh}
\begin{document}

\title{Reilly-type inequalities for $p$-Laplacian on submanifolds in space forms}

\author{Hang Chen}
\address[Hang Chen]{Department of Applied Mathematics, Northwestern Polytechnical University, Xi' an 710129, P. R. China \\ email: chenhang86@nwpu.edu.cn}
\thanks{Chen supported by NSFC Grant No. 11601426}
\author{Guofang Wei}
\address[Guofang Wei]{Department of Mathematics, University of California, Santa Barbara CA 93106 \\ email: wei@math.ucsb.edu}
\thanks{GW partially supported by NSF DMS 1506393}

\begin{abstract}
Let $M$ be an $n$-dimensional closed orientable submanifold in an $N$-dimensional space form. When $1<p \le \frac n2 + 1$, we obtain an upper bound for the first nonzero eigenvalue of the $p$-Laplacian in terms of the mean curvature of $M$ and the curvature of the space form. This generalizes the Reilly inequality for the Laplacian \cite{Rei77, ESI92} to the $p$-Laplacian and extends the work of \cite{DM15} for the $p$-Laplacian.
\end{abstract}



\maketitle

\section{Introduction}
Let $M$ be an $n$-dimensional compact Riemannian manifold. The $p$-Laplacian ($p>1$) is a second order quasilinear elliptic operator on $M$ defined by
\begin{equation}
\pl u=\dive (|\nabla u|^{p-2}u).
\end{equation}  
It is the usual Laplacian when $p=2$. Similar with Laplacian, one can consider the eigenvalue problem of $\pl$. A real number $\la$ is called a Dirichlet (or Neumann) eigenvalue if there exists a non-zero function $u$ satisfying the following equation with Dirichlet boundary condition  $u \equiv 0 $ on $\partial M$ (or  Neumann boundary condition $\nabla_{\nu}u \equiv 0$  on $\partial M$):
\begin{equation}
	\pl u=-\la |u|^{p-2} u\quad  \text{on $M$},
\end{equation}
where $\nu$ is the outward normal on $\partial M$.

So far, many estimates for the first eigenvalue of Laplacian have been generalized to $\pl$. For instance, 
 Matei \cite{Mat00} extended Cheng's first Dirichlet eigenvalue comparison of balls \cite{Che75} to the $p$-Laplacian. For closed Riemannian manifolds with Ricci curvature bounded below by $(n-1)K$,  a sharp lower bound for the first nontrivial eigenvalue of $p$-Laplacian was obtained by Matei \cite{Mat00} for $K>0$, Valtora \cite{Val12} for $K=0$ and Naber-Valtora \cite{NV14} for general real number $K$ respectively, see also \cite{And15,CM17}. Recently, Seto-Wei \cite{SW17} gave various estimates of the first eigenvalue of the $p$-Laplacian on closed Riemannian manifolds with integral curvature condition.

Now we consider a Riemannian manifold $M$ without boundary. The first nonzero eigenvalue of $\pl$, denoted by $\la_{1,p}$,  has a Rayleigh type variational characterization (cf. \cite{Ver91}):
\begin{equation}
\la_{1,p}=\inf\left\{ \frac{\int_M |\nabla u|^p}{\int_M|u|^p} \Big| u\in W^{1,p} (M)\backslash \{0\}, \int_M|u|^{p-2} u =0 \right\}.
\end{equation}
When $M$ is a submanifold immersed in $\F$, where $\F$ is the $N$-dimensional simply connected space form of constant sectional curvature $c$ and represents the Euclidean space $\R$, the unit
sphere $\S$ and the hyperbolic space $\H$ for $c=0, 1$ and $-1$ respectively, there is a well-known estimate for the first  nonzero eigenvalue of Laplacian.
\begin{thm}[cf. \cite{Rei77,ESI92}]\label{thm_Rei}
	Let $M$ be an $n$-dimensional  closed orientable submanifold in an $N$-dimensional space form $\F$. Then the first non-zero eigenvalue $\lambda_{1}^{\Delta}$ of Laplacian satisfies
\begin{equation}\label{Rei}
\lambda_1^{\Delta}\leq\frac{n}{\vol(M)}\int_M\big(|\mathbf{H}|^2+c\big),
\end{equation}
where $\mathbf{H}$ is the mean curvature vector of $M$ in $\F$, and $\vol(M)$ is the volume of $M$. Moreover, the equality holds in \eqref{Rei} if and only if $M$ is minimally immersed in a geodesic sphere of radius $r_c$ of $\F$  with $r_0=\big(n\big/\lambda_1^{\Delta}\big)^{1/2}, r_1=\arcsin r_0$ and $r_{-1}=\arcsinh r_0$. 
\end{thm}

 R. Reilly first proved \eqref{Rei} for $c=0$ in  \cite{Rei77}, hence  \eqref{Rei} is usually referred ``Reilly inequality". For the case $c=1$, by embedding the sphere $\S \rightarrow \mathbb R^{N+1}$, one can reduce it to the case $c=0$. In these case, one can move the center of $M$ to the origin and then the  coordinate functions can be used as test functions.  But this does not work for $c=-1$. The case $c=-1$   was proved by El Soufi and Ilias
in \cite{ESI92} by conformally mapping $\H$ to $\S$. 

Recently, F. Du and J. Mao \cite{DM15} extended the eigenvalue estimate to  the $p$-Laplacian for $c=0,1$. 

\begin{thm}[cf. Theorem 1.2 and Theorem 1.5 in \cite{DM15}]
		Let $M$ be an $n$-dimensional  closed orientable submanifold in an $N$-dimensional space form $\F$. For $p>1$, the first non-zero eigenvalue $\lambda_{1,p}$  of the $p$-Laplacian satisfies
\begin{numcases}
{\la_{1,p}\leq} N^{|2-p|/2}\frac{n^{p/2}}{\big(\vol(M)\big)^{p-1}}\Big(\int_M|\mathbf{H}|^{\frac{p}{(p-1)}}\Big)^{p-1}& for $c=0$, \label{Mao0}\\
(N+1)^{|2-p|/2}\frac{n^{p/2}}{\big(\vol(M)\big)^{p-1}}\Big(\int_M\big(1+|\mathbf{H}|^2\big)^{\frac{p}{2(p-1)}}\Big)^{p-1}& for $c=1$. \label{Mao1}
\end{numcases}
\end{thm}


 In this paper, we generalize the Reilly inequality \eqref{Rei} to the case of $p$-Laplacian for all $c$. 
 Our main theorem is the following.
\begin{thm}\label{thm1.1}
	Let $M$ be an $n(\geq 2)$-dimensional  closed orientable submanifold in an $N$-dimensional space form $\F$. Then the first non-zero eigenvalue  $\lambda_{1,p}$ of the $p$-Laplacian satisfies 
		\begin{numcases}
	{\la_{1,p}\leq}
	 (N+1)^{1-\frac{p}{2}}\frac{n^{p/2}}{\big(\vol(M)\big)^{p/2}}\Big(\int_M\big(c+|\mathbf{H}|^2\big)\Big)^{p/2} & for $1 < p\leq 2$,	\label{eq_thm1} \\
	 (N+1)^{\frac{p}{2}-1}
	 \frac{ n^{p/2}}{\vol(M)}\int_M(|c+|\mathbf{H}|^2|)^{p/2}& for  $2 < p \le \frac n2 +1$ \label{eq_thm2}.
	\end{numcases}
	Moreover, the equality holds if and only if $p=2$ and $M$ is minimally immersed in a geodesic sphere of radius $r_c$ of $\F$  with $r_0=\big(n\big/\lambda_1^{\Delta}\big)^{1/2}, r_1=\arcsin r_0$ and $r_{-1}=\arcsinh r_0$.  
\end{thm}


\begin{rem}
When $p=2$ our estimate recovers \eqref{Rei}. 
	When $1<p\leq 2$, we have $\frac{p}{2(p-1)}\geq 1$. Then by H\"{o}lder inequality, we have 
	\begin{align*}
		\int_M\big(c+|\mathbf{H}|^2\big)&\leq \big(\vol(M)\big)^{1-\frac{2(p-1)}{p}}\Big(\int_M\big(c+|\mathbf{H}|^2\big)^{\frac{p}{2(p-1)}}\Big)^{\frac{2(p-1)}{p}}.
	\end{align*}
	Hence, the upper bound in \eqref{eq_thm1} is better than \eqref{Mao1} for $c=1$.


	When $2<p\leq \frac{n}{2}+1$, \eqref{eq_thm2} is weaker than \eqref{Mao0} and \eqref{Mao1} for $c=0$ and $c=1$. But \eqref{eq_thm2} is  a new estimate for $c=-1$.
\end{rem}

\begin{rem}
Matei \cite{Mat13} gave an upper bound for $\lambda_{1,p}$ in terms of conformal volume  when $1<p\leq \dim M$. 
\end{rem}

This paper is organized as follows. In Section 2, we recall structure equations for a submanifold $M$ in $\F$. We also show how some geometric quantities change when the metric on $\F$ changes under the conformal transformation. In Section 3, we prove Theorem \ref{thm1.1}. To estimate the upper bound of $\lambda_{1,p}$, as the method in \cite{DM15} does not work for $c=-1$, we find suitable test functions by conformal transformation to a unit sphere.

\textbf{Acknowledgment}: 
  This work is done while the first author is visiting UCSB. He would like to thank UCSB math department for the hospitality, and he also would like to acknowledge financial support from China Scholarship Council and Top International University Visiting Program for Outstanding Young scholars of Northwestern Polytechnical University.

\section{Preliminaries and notations}
In this section, we recall some well-known facts on geometry of submanifolds and conformal geometry by using the moving frame method.  We use the following convention on the ranges of indices except special declaration:
\begin{equation*}
1\leq i, j, k, \ldots \leq n;\quad
n+1\leq\alpha, \beta, \gamma, \ldots \leq N;\quad
1\leq A, B, C, \ldots \leq N.
\end{equation*}

\subsection{Structure equations for submanifolds}
Let $x$ be the immersion from $M^n$ to an  $N$-dimensional Riemannian manifold $(\bar{M},\bar{g})$. Then $M$ has an induced metric $g_M=x^*\bar{g}$.

We denote the Levi-Civita connections on $M$ and $\bar{M}$ by $\nabla$ and $\bar{\nabla}$ respectively. Choose an orthonormal frame $\{e_A\}_{A=1}^{N}$ on $\bar{M}$ such that $\{e_i\}_{i=1}^n$ are tangent to $M$ and $\{e_{\alpha}\}_{\alpha=n+1}^{N}$ are normal to $M$.  Let $\{\omega_A\}_{A=1}^{N}$ be the dual coframe of $\{e_A\}_{A=1}^{N}$. Then the structure equations of $\bar{M}$ are (cf. \cite{Chern1968}):
\begin{equation}
\left\{\begin{aligned}\label{eq21}
& d\omega_{A}=\sum_{B}\omega_{AB}\wedge\omega_{B},\quad \omega_{AB}+\omega_{BA}=0, \\
&d\omega_{AB}-\sum_C\omega_{AC}\wedge\omega_{CB}=-\frac{1}{2}\sum_{C,D}\bar{R}_{ABCD}\,\omega_C\wedge\omega_D,
\end{aligned}\right.
\end{equation}
where $\{\omega_{AB}\}$ are the connection forms on $\bar{M}$, and  $\bar{R}_{ABCD}$ are components of the curvature tensor of  $\bar{M}$. 

Denote $x^{\ast}\omega_{A}=\theta_{A}, x^{\ast}\omega_{AB}=\theta_{AB}$, then restricted to $M$, we have (cf. \cite{Chern1968})
\begin{equation}\label{eq23}
\theta_{\alpha}=0, \quad\theta_{i\alpha}=\sum_j h_{ij}^{\alpha}\theta_{j}.
\end{equation}
and
\begin{equation}
\left\{\begin{aligned}\label{eq22}
&d\theta_{i}=\sum_{j}\theta_{ij}\wedge\theta_{j}, \quad\theta_{ij}+\theta_{ji}=0,\\
& d\theta_{ij}-\sum_k\theta_{ik}\wedge\theta_{kj}=-\frac{1}{2}\sum_{k,l}R_{ijkl}\,\theta_k\wedge\theta_l;
\end{aligned}
\right.
\end{equation}
where $R_{ijkl}$ are components of the curvature tensor of $M$, and $h_{ij}^{\alpha}$ are components of the second fundamental form of $M$ in $\bar{M}$.


We take $\bar{M}=\F$, then $\bar{R}_{ABCD}=c(\delta_{AC}\delta_{BD}-\delta_{AD}\delta_{BC})$. Pulling back \eqref{eq21} by $x$ and using \eqref{eq23} and \eqref{eq22}, we obtain 
the Gauss equations
\begin{align}
\label{gauss}
R_{ijkl}&=(\delta_{ik}\delta_{jl}-\delta_{il}\delta_{jk})c+\sum_\alpha( h_{ik}^{\alpha}h_{jl}^{\alpha}-h_{il}^{\alpha}h_{jk}^{\alpha}),\\
R&=n(n-1)c+n^2H^2-S,\label{gauss3}
\end{align}
where $R$ is the scalar curvature of $M$, $S=\sum\limits_{\alpha,i,j}(h_{ij}^{\alpha})^2$ is the
norm square of the second fundamental form, $\mathbf{H}=\sum\limits_\alpha H^{\alpha}e_\alpha=\frac{1}{n}
\sum\limits_\alpha(\sum\limits_i h_{ii}^{\alpha})e_\alpha$ is the mean curvature  vector of $M$, and $H=|\mathbf{H}|$.

\subsection{Conformal relations} In this subsection, we focus on how curvature and the second fundamental form change under the conformal transformation. Although these relations are well-known (cf. \cite{Che73a,Che74}),  we give a brief proof for readers' convenience,  by using the moving frame method.

Now assume that $\bar{M}$ is equipped with a new metric $\tbg=e^{2\rho}\bar{g}$ which is conformal to $\bar{g}$, where $\rho\in C^{\infty}(\bar{M})$. Then $\{\et_A=e^{-\rho}e_A\}$ is an orthonormal frame of $(\bar{M},\tbg)$, and $\{\tom_A=e^{\rho}\omega_A\}$ is the dual coframe of  $\{\et_A\}$. The structure equations of $(\bar{M},\tbg)$ are given by

\begin{equation}
\left\{\begin{aligned}\label{eq25}
&d\tom_{A}=\sum\limits_{B}\tom_{AB}\wedge\tom_{B},\\
&\tom_{AB}+\tom_{BA}=0,
\end{aligned}\right.
\end{equation}
where $\{\tom_{AB}\}$ are the connection forms on $(\bar{M},\tbg)$.
 Denoting $\tg_M=x^{\ast}\tbg, \ x^{\ast}\tom_{A}=\tt_{A}, \ x^{\ast}\tom_{AB}=\tt_{AB}$, then restricted to $(M,\tg_M)$, we have
\begin{equation}\label{eq29}\tt_{\alpha}=0, \quad\tt_{i\alpha}=\sum_j \th_{ij}^{\alpha}\tt_{j},
\end{equation}
and
\begin{equation}
\left\{
\begin{aligned}\label{eq28}
&\displaystyle d\tt_{i}=\sum_{j}\tt_{ij}\wedge\tt_{j}, \quad\tt_{ij}+\tt_{ji}=0, \\
\displaystyle &d\tt_{ij}-\sum_k\tt_{ik}\wedge\tt_{kj}=-\frac{1}{2}\sum_{k,l}\tR_{ijkl}\,\tt_k\wedge\tt_l;
\end{aligned}
\right.\end{equation}
where $\tR_{ijkl}$ are components of the curvature tensor of $(M, \tg_M)$ and $\th_{ij}^{\alpha}$ are components of the second fundamental form of $(M, \tg_M)$ in $(\bar{M}, \tbg)$.

From (\ref{eq21}) and (\ref{eq25}), we can solve 
\begin{equation}
\begin{aligned}\label{eq26}
\tom_{AB}=\omega_{AB}+\rho_A\omega_B-\rho_B\omega_A,
\end{aligned}
\end{equation}
where $\rho_A$ is the covariant derivative of $\rho$ with respect to $e_A$, i.e. $d\rho=\sum_A\rho_Ae_A$. 

We derive from \eqref{eq22}, \eqref{eq26} and \eqref{eq28}
\begin{align}
e^{2\rho}\tR_{ijkl}=&R_{ijkl}-(\rho_{ik}\delta_{jl}+\rho_{jl}\delta_{ik}-\rho_{il}\delta_{jk}-\rho_{jk}\delta_{il})\nonumber\\
&+(\rho_i\rho_k\delta_{jl}+\rho_j\rho_l\delta_{ik}-\rho_j\rho_k\delta_{il}-\rho_i\rho_l\delta_{jk})\nonumber\\
&-|\nabla\rho|^2(\delta_{ik}\delta_{jl}-\delta_{il}\delta_{jk}).\label{eq209}
\end{align}

By pulling back (\ref{eq26}) to $M$ by $x$ and using (\ref{eq23}) and (\ref{eq29}), we have
\begin{equation}\label{eq208}
\th_{ij}^{\alpha}=e^{-\rho}(h_{ij}^{\alpha}-\rho_{\alpha}\delta_{ij}), \quad\tH^{\alpha}=e^{-\rho}(H^{\alpha}-\rho_{\alpha}),
\end{equation}
from this, it is easy to show  the well-known relation 
\begin{equation}\label{eq_um}
e^{2\rho}(\tilde{S}-n\tH^2)=S-nH^2.
\end{equation}

\section{Proof of Theorem \ref{thm1.1}}
In this section, we prove our main result, Theorem \ref{thm1.1}. At first, we recall some  lemmas from \cite{Mat13} to our setting.  For convenience of the reader, we also give the proof here as we need the proof to analyse the equality case.
\begin{lem}[cf. Lemma 2.6 in \cite{Mat13}]\label{lem3.1}
Let $x: M\to \F$ be the immersion from an $n$-dimensional closed orientable submanifold to an $N$-dimensional space form $\F$. Then for $p>1$, there exists a regular conformal map $\Gamma: \F\to \S\subset\mathbb{R}^{N+1}$ such that  the immersion $\Phi=\Gamma\circ x=(\Phi^1, \cdots, \Phi^{N+1})$ satisfies that
\begin{equation}\label{eq_lem31}
\int_M |\Phi^A|^{p-2}\Phi^A\,dv_M=0,~A=1,\ldots,N+1.
\end{equation}
\end{lem}

\begin{proof}
	The main idea of Lemma \ref{lem3.1} is inspired by the case $p=2$ (cf. \cite{LY82,ESI00}). 
	  
	First observe that there is the standard conformal map $\Pi_c$ from $\F$ to $\S$. Here $\Pi_c$ is identity when $c=1$, and $\Pi_c$ can be given by the stereographic projection when $c=0$ or $c=-1$.
	
	For any $a\in \S$, consider the flow  $\gamma_t^a$ generated by the  vector field $V_a(x)=a-\l x,a\r x$ on $\S$. In fact, $\gamma_t^a=\pi_a^{-1}(e^t\pi_a(x)), x\in \S$, where $\pi_a$ is the stereographic projection of pole $a$. It is easy to see that  $\gamma_0^a$ is identity map on $\S$, and $\gamma_t^a(x)\to a$ for any $x\in\S$ as $t\to +\infty$.
	
	We claim there is a $\gamma_t^a$ such that $\Gamma=\gamma_t^a\circ \Pi_c$ satisfies the required property \eqref{eq_lem31}. If not, we can define a map $F(t,a): [0,+\infty) \times \S\to \S$ as follows:
	
	\begin{equation}
		F(t,a)=\frac{(\int_M |\Phi^1|^{p-2}\Phi^1,\cdots,\int_M |\Phi^{N+1}|^{p-2}\Phi^{N+1})}{\|(\int_M |\Phi^1|^{p-2}\Phi^1,\cdots,\int_M |\Phi^{N+1}|^{p-2}\Phi^{N+1})\|}.
	\end{equation}
	
	Now $F(0,\cdot)$ maps any  $a\in \S$ to a fixed point in $\S$. And $F(+\infty, \cdot)$ maps $a=(a^1,\cdots, a^{N+1})\in\S$ to $\frac{(|a^1|^{p-2}a^1,\cdots,  |a^{N+1}|^{p-2} a^{N+1})}{\|(|a^1|^{p-2}a^1,\cdots,  |a^{N+1}|^{p-2} a^{N+1})\|}$, which is bijective. So $\deg F(0,\cdot)=0$ and $\deg F(+\infty,\cdot)$ is odd. But $F(\cdot,\cdot)$ gives a homotopy between  $F(0,\cdot)$ and $F(+\infty,\cdot)$, which is a contradiction. So we complete the proof.
\end{proof}
 
 	Using the test function constructed in the above lemma, we can get an upper bound for $\lambda_{1,p}$  in terms of the conformal function, compare Lemma 2.7 in \cite{Mat13}. 
\begin{lem}\label{lem3.2}
 	Let $M$ be an $n(\geq 2)$-dimensional  closed orientable submanifold in an $N$-dimensional space form $\F$.  Denote by $h_c$  the standard metric on $\F$  and assume $\Gamma^{\ast}h_1=e^{2\rho}h_c$, where $\Gamma$ is the conformal map in Lemma \ref{lem3.1}. Then we have, for all $p>1$, 
 	\begin{equation}\label{eq3.3}
 		\lambda_{1,p}\vol(M)\leq (N+1)^{|1-p/2|}n^{p/2}\int_M(e^{2\rho})^{p/2}.
 	\end{equation}
 \end{lem}
 \begin{proof}
  	By Lemma \ref{lem3.1}, we can choose $\Phi^A$ as the test function, so 
 	\begin{equation}\label{eq3.7}
 	\la_{1,p}\int_M|\Phi^A|^p\leq|\nabla\Phi^A|^p, \quad 1\leq A\leq N+1.
 	\end{equation}
 	Note that $\sum_{A=1}^{N+1}|\Phi^A|^2=1$, then $|\Phi^A|\leq 1$. We also have 
 	\begin{equation}\label{eq32}
 	\sum_{A=1}^{N+1}|\nabla\Phi^A|^2=\sum_{i=1}^{n}|\nabla_{e_i}\Phi|^2=ne^{2\rho}.
 	\end{equation}
 	
 	When $1<p\leq 2$, we have
 	\begin{equation}\label{eq3.8}
 	|\Phi^A|^2\leq |\Phi^A|^p,
 	\end{equation}
 	Then by using \eqref{eq3.7}, \eqref{eq32}, \eqref{eq3.8} and the H\"{o}lder inequality, we have
 	\begin{align*}
 	\la_{1,p}\vol(M)&=\la_{1,p}\sum_{A=1}^{N+1}\int_M|\Phi^A|^2\\
 	&\leq\la_{1,p}\sum_{A=1}^{N+1}\int_M|\Phi^A|^p\leq\int_M\sum_{A=1}^{N+1}|\nabla\Phi^A|^p\\
 	&\leq (N+1)^{1-p/2}\int_M\Big(\sum_{A=1}^{N+1}|\nabla\Phi^A|^2\Big)^{p/2}\\
 	&=(N+1)^{1-p/2}\int_M(ne^{2\rho})^{p/2}.
 	\end{align*}
 This is \eqref{eq3.3}.

 When $p\geq 2$, the H\"{o}lder inequality gives
 \begin{equation}\label{eq3-6}
 	1=\sum_{A=1}^{N+1}|\Phi^A|^2\leq (N+1)^{1-\frac 2p}(\sum_{A=1}^{N+1}|\Phi^A|^p)^{2/p},
 \end{equation}
 from which we have
 \begin{equation}\label{eq3-7}
 	\lambda_{1,p}\vol(M)\leq (N+1)^{\frac p2-1}\Big(\sum_{A=1}^{N+1}\lambda_{1,p}\int_M|\Phi^A|^p\Big).
 \end{equation}
On the other hand, we have 
\begin{equation}\label{eq3-8}
	\sum_{A=1}^{N+1}|\nabla\Phi^A|^p\leq(\sum_{A=1}^{N+1}|\nabla\Phi^A|^2)^{p/2}=(ne^{2\rho})^{p/2}.
\end{equation}
Hence \eqref{eq3.3} follows from \eqref{eq3.7}, \eqref{eq3-7} and \eqref{eq3-8}.
 \end{proof}

%

\begin{proof}[Proof of Theorem \ref{thm1.1}]
When $1<p \le 2$,  then $p/2\leq 1$. Using Lemma \ref{lem3.2}  and  the  H\"{o}lder inequality,  we have 
\begin{align*}
\la_{1,p}\vol(M)& \le (N+1)^{1-p/2}n^{p/2}\int_M(e^{2\rho})^{p/2}\\
&\leq (N+1)^{1-p/2}n^{p/2}\big(\vol(M)\big)^{1-p/2}\Big(\int_Me^{2\rho}\Big)^{p/2}.
\end{align*}

Note that we can compute $e^{2\rho}$ using the conformal relations and Gauss equations as follows. 	We take $\bar{M}=\F,\ \tg=h_c, \ \tbg=\Gamma^*h_1$ in Subsection 2.2. 
From \eqref{gauss3}, the Gauss equations for the immersion $x$ and the immersion $\Phi=\Gamma\circ x$ are respectively:
\begin{align}
R=&n(n-1)c+n(n-1)H^2+(nH^2-S),\label{G1}\\
\tR=&n(n-1)+n(n-1)\tH^2+(n\tH^2-\tilde{S}).\label{G2}
\end{align}
Contracting \eqref{eq209} we have
\begin{equation}\label{eq_scal}
e^{2\rho}\tR=R-(n-2)(n-1)|\nabla \rho|^2-2(n-1)\Delta\rho.
\end{equation}

Now from \eqref{G1}, \eqref{G2},  \eqref{eq_scal},  \eqref{eq208} and  \eqref{eq_um}, we derive
\begin{align*}
&n(n-1)(e^{2\rho}-c)+n(n-1)\sum_{\alpha}(H^\alpha-\rho_{\alpha})^2-n(n-1)H^2\\
=&-(n-2)(n-1)|\nabla \rho|^2-2(n-1)\Delta\rho,
\end{align*}
divided by $n(n-1)$ on both sides, we obtain 	\begin{equation}\label{eq_relation}
	e^{2\rho}=(c+|\mathbf{H}|^2)-\frac{2}{n}\Delta \rho-\frac{n-2}{n}|\nabla\rho|^2-|(\bar{\nabla}\rho)^\bot-\mathbf{H}|^2.
\end{equation}
Hence, 
\begin{align*}
\la_{1,p}\vol(M)
&\leq (N+1)^{1-p/2}n^{p/2}\big(\vol(M)\big)^{1-p/2}\Big(\int_Me^{2\rho}\Big)^{p/2}\\
&\leq (N+1)^{1-p/2}\frac{n^{p/2}}{\big(\vol(M)\big)^{p/2-1}}\Big(\int_M\big(c+|\mathbf{H}|^2\big)\Big)^{p/2},
\end{align*}
which is equivalent to \eqref{eq_thm1}.

	When $p>2$, we cannot use $\int_M(e^{2\rho})$ to control $\int_M(e^{2\rho})^{p/2}$ by applying H\"{o}lder inequality directly. 
	Instead multiplying $e^{(p-2)\rho}$ on both sides of \eqref{eq_relation}, and then integrating on $M$ (cf. \cite{CL11}), we obtain
\begin{equation}\label{eq4.1}
\int_Me^{p\rho}\leq\int_M(c+|\mathbf{H}|^2)e^{(p-2)\rho}-\int_M\frac{n-2-2(p-2)}{n}e^{(p-2)}|\nabla\rho|^2\leq\int_M(c+|\mathbf{H}|^2)e^{(p-2)\rho}.
\end{equation}
where we used the assumption $n\geq 2p-2$.

On the other hand, by Young's inequality, we have 
\begin{equation}\label{eq4.2}
\int_M(c+|\mathbf{H}|^2)e^{(p-2)\rho}\leq\frac{1}{p/2}\int_M(|c+|\mathbf{H}|^2|)^{p/2}+\frac{1}{p/(p-2)}\int_Me^{p\rho}.
\end{equation}
Hence we obtain 
\begin{equation}\label{eq4.3}
	\int_Me^{p\rho}\leq\int_M(|c+|\mathbf{H}|^2|)^{p/2}.
\end{equation}
from \eqref{eq4.1} and \eqref{eq4.2}. Putting \eqref{eq4.3} into \eqref{eq3.3}, we obtain \eqref{eq_thm2}.

Now we check the equality case. When the equality holds in \eqref{eq_thm1},  by checking \eqref{eq3.7} and \eqref{eq3.8},  we must have $|\Phi^A|^2=|\Phi^A|^p$ 
and $\pl \Phi^A=-\lambda_{1,p} |\Phi^A|^{p-2}\Phi^A$ 
for each $A=1,\cdots, N+1$. If $1<p<2$, then $|\Phi^A|=0$ or $1$. But $\sum_{A=1}^{N+1}|\Phi^A|^2=1$,  so there is exactly one $A$ such that $|\Phi^A|=1$ and then $\lambda_{1,p}=0$, which is a contradiction.  Hence $p=2$ and it reduces to the Laplacian case. We can complete the proof by using Theorem \ref{thm_Rei}.

When equality holds in \eqref{eq_thm2}, suppose $p>2$,  then \eqref{eq3-7} and \eqref{eq3-8} must become the equalities, which means
\begin{equation*}
	|\Phi^1|^p=\cdots=|\Phi^{N+1}|^p,
\end{equation*}
and there exists some $A$ such that $|\nabla\Phi^A|=0$. So $\Phi^A$ is constant and then $\lambda_{1,p}=0$, which is a contradiction. 
\end{proof}
%

\end{document}